\documentclass[12pt]{article}

\usepackage{amsmath, amssymb, amsthm}
\usepackage{graphicx}
\usepackage{tikz}
\usepackage{dsfont}
\usepackage{hyperref}
\usepackage{enumitem}
\usepackage{mathtools}
\usepackage{geometry}
\usepackage{titlesec}
\titleformat{\section}{\Large\bfseries}{\thesection.}{0.6em}{}
\titleformat{\subsection}{\large\bfseries}{\thesubsection}{0.6em}{}

\theoremstyle{plain}
\newtheorem{theorem}{Theorem}[section]
\newtheorem{lemma}[theorem]{Lemma}
\newtheorem{proposition}[theorem]{Proposition}
\newtheorem{corollary}[theorem]{Corollary}

\theoremstyle{definition}
\newtheorem{definition}[theorem]{Definition}
\newtheorem{example}[theorem]{Example}

\theoremstyle{remark}

\newcommand{\cF}{{\cal{F}}}

\newcommand{\fusc}{{\rm fusc}}

\newcommand{\QQ}{{\mathds Q}}

\newcommand{\llb}{[\![}
\newcommand{\rrb}{]\!]}

\title{\textbf{
Dimers, filters, and $q$-deformed real numbers}}
\author{James Propp}
\date{}

\begin{document}
\maketitle

\begin{abstract}
\noindent
This article associates to each positive real number $x$
an inhomogeneous dimer model with an activity parameter $q>0$
and uses it to define a positive-real-valued function $q \mapsto \llb x \rrb_q$
that is related to (and indeed was inspired by) 
the algebraic $q$-deformation $[x]_q$
introduced by Morier-Genoud and Ovsienko~\cite{MO22}.
The technical details are most transparent
when the dimer model is replaced by an equivalent model 
involving filters in partially ordered sets.
When $x$ is rational we have $\llb x \rrb_q = q \: [x]_q$,
and it seems likely that more generally
the two sides of the equation agree wherever both are defined. 

\end{abstract}

\section{Introduction}
\label{sec:intro}

The $q$-rationals and $q$-reals were introduced by Morier-Genoud and Ovsienko
in a recent series of papers \cite{MO20}, \cite{MO22}, \cite{MO25} by means
of two constructions.  Their first construction associates to each rational 
number $r>0$ a rational function $[r]_q$ in $\QQ(q)$ that generalizes the 
classical $q$-integer $[n]_q = 1+q+q^2+\cdots+q^{n-1}$. These rational 
functions arise naturally from continued fractions and possess a number of 
remarkable properties.  Most strikingly, one can then construct $q$-adic 
limits of $q$-deformed rational numbers to obtain $q$-deformations $[x]_q$ 
of irrational numbers $x$ as Laurent series with integer coefficients. 

The purpose of this paper is to take a different sort of deformation of 
real numbers via probability theory.
For each positive real number $x$, I define 
an edge-weighted bipartite planar graph $G_x(q)$, 
finite or infinite according to whether $x$ is rational or irrational,
with some edges assigned weight $q>0$ and all others assigned weight 1.
$G_x(q)$ also has a special edge $e$.  
Recall that a dimer cover of a graph 
(aka {\em perfect matching} or {\em matching} for short) 
is a collection of edges having the property that 
every vertex of the graph belongs to exactly one edge in the collection.

In the case in which $x$ is rational, I define the weight of a dimer cover 
of the graph $G_x(q)$ as the product of the weights 
of its constituent edges. I let $\mu_{x,q}$ be the probability distribution 
in which each matching has probability proportional to its weight,
and I define $\llb x \rrb_q$ as the odds 
that a $\mu_{x,q}$-random matching of $G_x(q)$ includes $e$.
That is, $\llb x \rrb_q$ equals 
the probability that a random matching of $G_x(q)$ includes $e$
divided by 
the probability that a random matching of $G_x(q)$ omits $e$.
(Throughout this paper, when an event has probability $p$,
I will use the phrase ``the odds in favor of the event \dots'',
or ``the odds of the event \dots'', or ``the odds that \dots'', 
to signify the ratio $p/(1-p)$.)

In the case in which $x$ is irrational, I define $\mu_{x,q}$ 
as a limit of $\mu_{r,q}$ with $r$ rational and approaching $x$.
The main technical result of the paper is that this limit, 
suitably defined, exists for any sequence of rationals approaching $x$
and that this value is independent of the approximating rationals.
I can then define $\llb x \rrb_q$ once again as the odds
that a $\mu_{x,q}$-random matching of $G_x(q)$ includes $e$.

\begin{figure}[h!]
\begin{center}
\includegraphics[width=2.5in]{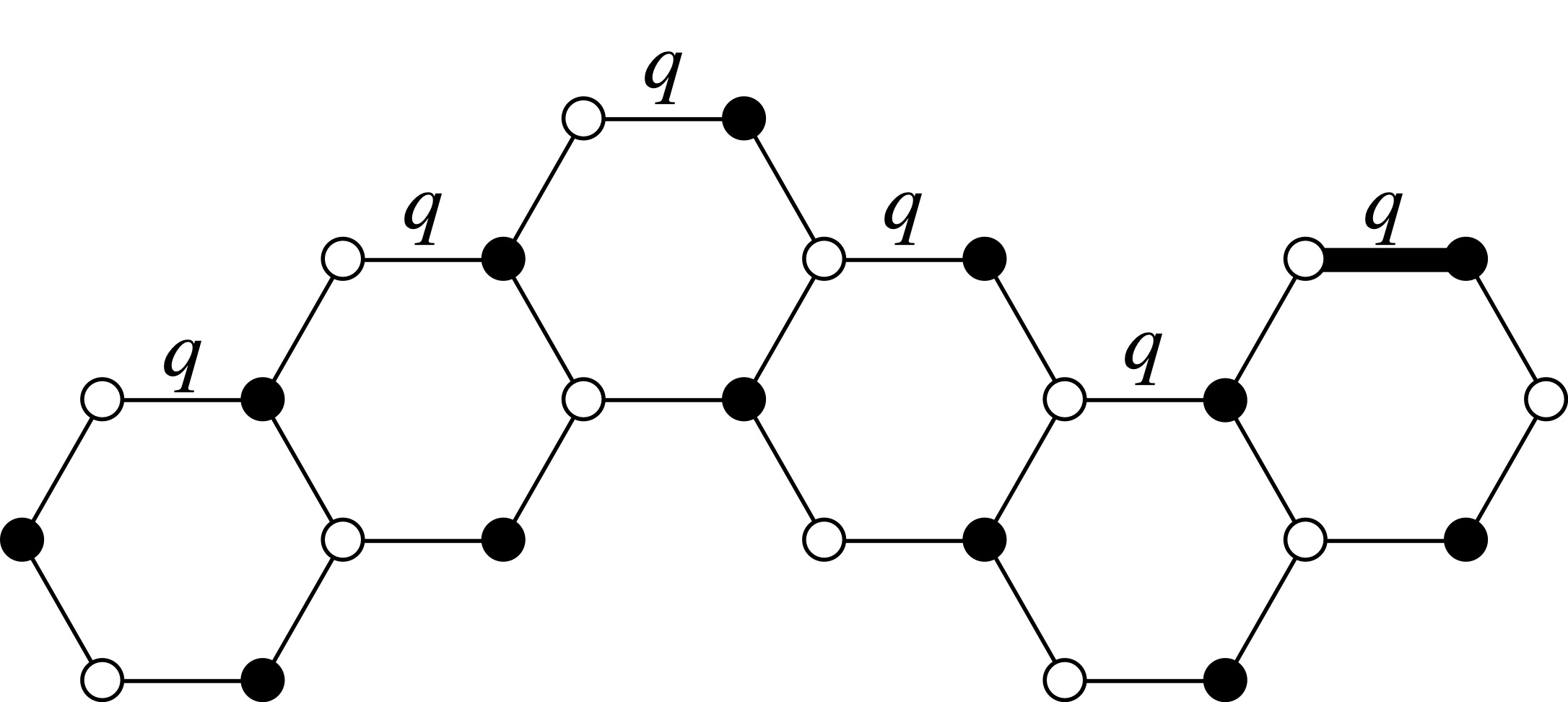}
\end{center}
\caption{The weighted hexagon snake graph $G_{10/7}(q)$.}
\label{fig:hex-fifty-seven}
\end{figure}

\begin{example}
Figure~\ref{fig:hex-fifty-seven} shows the 
finite graph $G_{10/7}(q)$ with special edge $e$ shown in bold. Each dimer 
cover of $G_{10/7}$ is assigned probability proportional to the product 
of the weights of its constituent edges, where edges marked ``$q$'' 
have weight $q>0$ and unmarked edges have weight 1. Let $Z$ be the sum 
of the weights of all the dimer covers of $G_{10/7}(q)$. The probability 
that a $\mu_{10/7,\,q}$-random dimer cover of $G_{10/7}(q)$ includes $e$ is 
$(q+q^2+2q^3+3q^4+2q^5+q^6)/Z$ while the probability that a 
$\mu_{10/7,\,q}$-random dimer cover of $G_{10/7}(q)$ omits $e$ is 
$(1+q+2q^2+2q^3+q^4)/Z$. $\llb 10/7 \rrb_q$ equals the ratio of 
the former probability to the latter. 
\end{example}

\begin{figure}[h!]
\begin{center}
\includegraphics[width=4.5in]{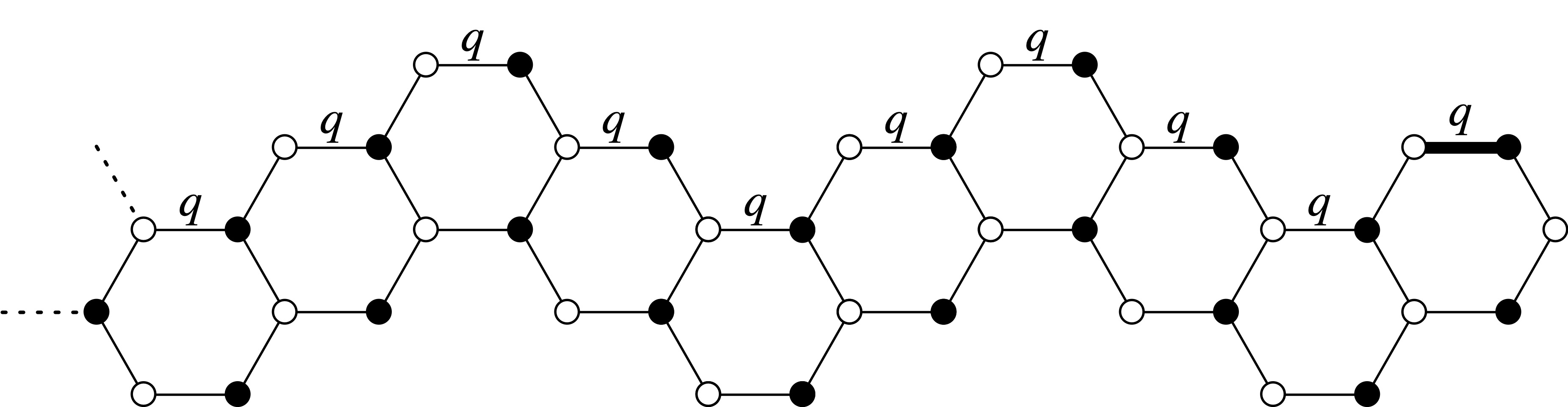}
\end{center}
\caption{The weighted hexagon snake graph $G_{\sqrt{2}}(q)$.}
\label{fig:hex-sqrt-two}
\end{figure}

\begin{example}
Figure~\ref{fig:hex-sqrt-two} shows the 
infinite graph $G_{\sqrt{2}}(q)$ with special edge $e$ shown in bold. 
As before, unmarked edges have weight 1.
Let $\mu_{\sqrt{2},\,q}$ be the measure on the set of dimer covers of 
$G_{\sqrt{2}}(q)$ that is obtained as a limit of measures associated with 
values of $r$ that converge to $\sqrt{2}$ (any such sequence of rationals 
will do).  $\llb \sqrt{2} \rrb_q$ equals 
the probability that a $\mu_{\sqrt{2},\,q}$-random dimer cover of 
$G_{\sqrt{2}}(q)$ includes $e$ divided by the probability that a 
$\mu_{\sqrt{2},\,q}$-random dimer cover of $G_{\sqrt{2}}(q)$ omits $e$. 
\end{example}

It can be shown (see Proposition~\ref{prop:q}) that 
when $x$ is a positive rational number, 
$\llb x \rrb_q$ is equal to $q \: [x]_q$,
where $[x]_q$ is the deformation of $x$ defined by Morier-Genoud and Ovsienko.
The factor of $q$ that distinguishes my $q$-deformation of $x$
from that of Morier-Genoud and Ovsienko
is merely cosmetic, but there is a deeper distinction: 
$\llb x \rrb_q$ is well-defined for all $q>0$ but not for
complex values of $q$, whereas $[x]_q$ is defined for complex values of
$q$ but only those that lie in a suitable disk centered at 0. Indeed,
Etingof~\cite{E25} has shown that when $x$ is irrational,
$[x]_q$ cannot extend meromorphically to 
any open set properly containing the open unit disk.
It is natural to hope that $\llb x \rrb_q$ and $q \: [x]_q$ can unified to give
a function that is holomorphic on a neighborhood of the positive real
ray in the complex plane, but it is unclear how to achieve this.

I will do most of the work in the context of filters of posets 
(partially ordered sets), since the resulting one-dimensional models 
are geometrically simpler than planar dimer models. For instance,
the hexagon snake graph of Figure~\ref{fig:hex-fifty-seven} is replaced 
by the {\em snake poset} shown in Figure~\ref{fig:snake-fifty-seven}.
After developing results for finite and infinite snake posets
I translate the poset construction into the language of dimer covers.

An outline of the paper follows.

In Section~\ref{sec:snakes}, I give an exposition
of Stern's diatomic sequence in terms of posets and filters
and use it to give a probabilistic interpretation
of the Calkin-Wilf enumeration of the positive rationals
(see Figure~\ref{fig:snakes}).
Specifically, given a positive rational $r$,
one uses the binary expansion of an associated positive integer $n$
to construct a finite poset $S_r$ (a ``snake poset'')
such that $r$, the $n$th Calkin-Wilf rational, 
is equal to the probability that a uniform random filter of $S_r$ includes 0 
divided by the complementary probability 
that a uniform random filter of $S_r$ omits 0. 
Much of this material appears elsewhere but with different conventions,
so it is included to keep the article self-contained.

In Section~\ref{sec:q}, I bring $q$ into the story,
replacing the uniform distribution on filters 
studied in Section~\ref{sec:snakes}
by the one-parameter family of weighted distributions $\mu_{r,q}$ 
in which the probability of a filter is proportional to
$q$ to the power of its cardinality.
Like the preceding section,
this section does not present genuinely new results,
but the results are less well-known than
those of Section~\ref{sec:snakes},
and readers are likely to appreciate
the explicit discussion of transfer-matrix machinery.

In Section~\ref{sec:irrational}, I bring positive irrational numbers 
into the story, replacing the finite snake posets 
of sections~\ref{sec:snakes} and~\ref{sec:q}
by countably infinite snake posets $S_x$;
this requires taking limits whose well-definedness is not immediately apparent.

In Section~\ref{sec:golden}, I apply the results of the preceding section
to the particular number $(1+\sqrt{5})/2$.

In Section~\ref{sec:dimer}, I relate filters in snake posets
to dimer covers in snake graphs as defined in~\cite{P20}.
This reinterpretation makes tacit use of 
the lattice structure for dimer covers 
of bipartite plane graphs described in \cite{P25}
but applied in reverse, so that one reconstructs a bipartite graph $G$
from the lattice structure on perfect matchings of $G$.

Section~\ref{sec:future} contains assorted remarks and suggestions
for directions of future work.

Few of the ingredients of this paper are new.
Other authors including Morier-Genoud and Ovsienko
have used finite snake posets before;
other authors~\cite{AL25} have used a single poset 
as the shared site of the two counting problems
associated with the numerator and denominator of $[r]_q$
(for $r$ rational) and have interpreted those enumerations
in terms of perfect matchings of finite snake graphs; see~\cite{O25}.
But as far as this author knows,
nobody has brought infinite snakes or hexagonal snakes into the picture
in connection with $[x]_q$ for irrational $x$.

A word about terminology: There is some disagreement in the literature 
regarding the proper meaning of the term 
``snake posets'' (aka ``fence posets'' and ``zigzag posets''). 
The contested issue is, must the zigs and zags strictly alternate?
Some authors say yes, and therefore use the term 
``generalized snake posets'' to describe the broader notion; 
others say that the word ``snake'' already permits loose alternation.
Since the phrase ``snake graph'' was coined over two decades ago
in the preprint version of~\cite{P20} and 
since that article allowed loose alternation, 
I will follow the same convention here,
and will use the term ``snake posets'' to denote
what some authors call ``generalized snake posets''.

\section{Positive rational numbers}
\label{sec:snakes}

\begin{figure}[h!]
\begin{center}
\includegraphics[width=4.5in]{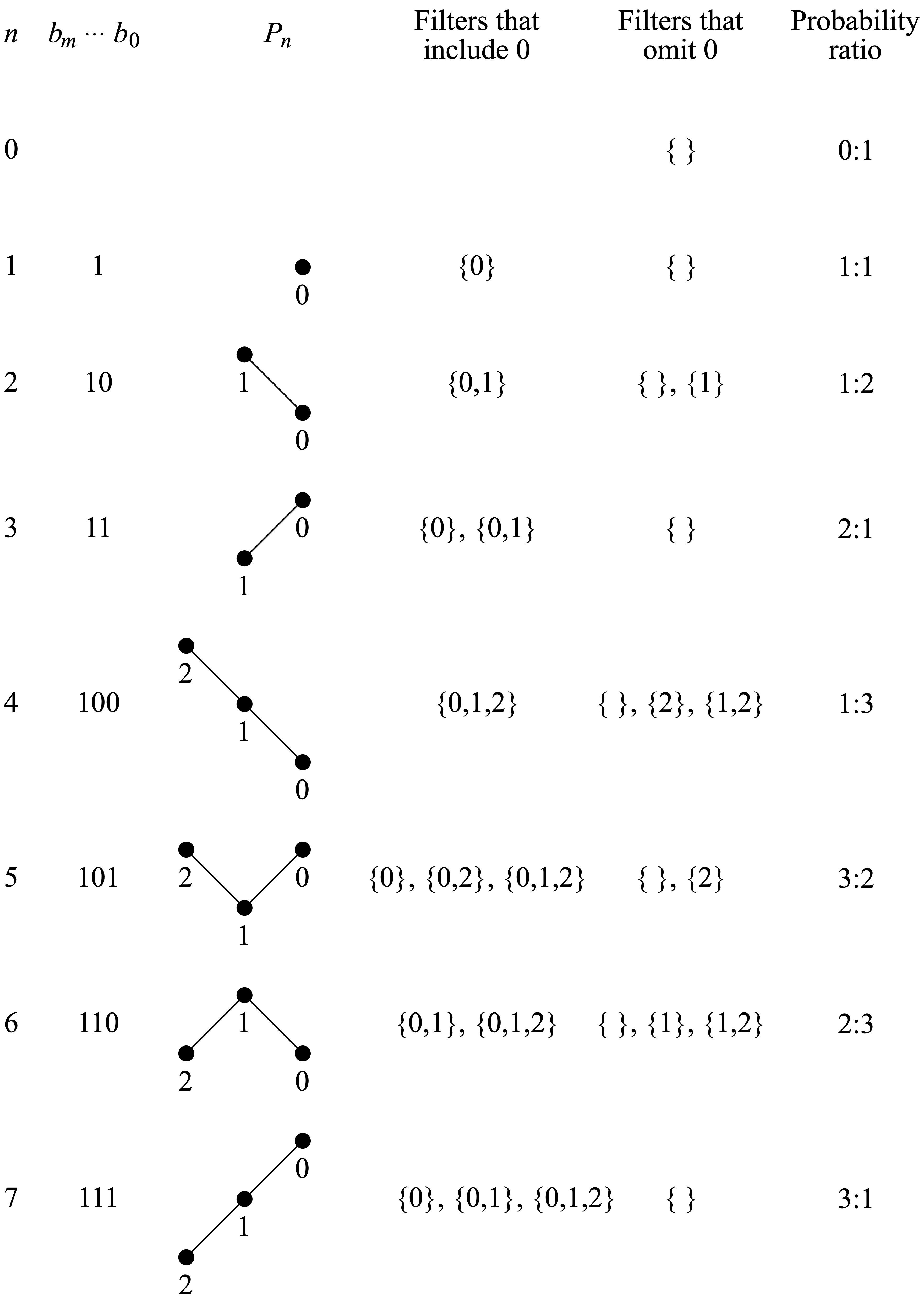}
\end{center}
\caption{The snake posets $P_n$ with $0 \leq n \leq 7$.}
\label{fig:snakes}
\end{figure}
A \emph{snake poset} $P$ is a poset 
whose Hasse diagram is a finite or infinite path;
in this article all infinite paths
will be half-infinite, i.e., will have one endpoint,
drawn on the right.
In the finite case, a snake poset is one
whose elements can be indexed as $v_0, v_1, v_2, \dots, v_m$ 
such that for each $0 \le i < m$,
$v_{i+1} \lessdot v_i$ or $v_{i+1} \gtrdot v_i$,
where no further relations hold besides the ones implied by these.
The infinite case is much the same, but without the cutoff at $m$.
(I assume that the reader has some basic acquaintance 
with poset theory, and for instance already knows that 
$\gtrdot$ and $\lessdot$ mean ``covers'' and ``is covered by'', respectively;
readers unacquainted with these notions should consult~\cite{St11}.)
For brevity I will write $v_i$ as just $i$,
but I stress that in this context one can have $i+1 \lessdot i$.

\begin{definition}
Let $P_0$ be the empty poset,
let $P_1$ be the poset with one element,
and when $n$ is a positive integer greater than 1
define $P_n$ as follows.
Write the binary representation of $n$ as 
$b_m \, b_{m-1} \, \cdots b_1 \, b_0$
(with $m = \lfloor \log_2 n \rfloor$,
$b_m = 1$, and $b_i \in \{0,1\}$ for all $i<m$,
with $\sum_{i \leq m} b_i 2^i = n$).
Then let $P_n$ be the poset on $0,1,\dots,m$
with $i+1 \gtrdot i$ (resp.\ $i+1 \lessdot i$)
if $b_i = 0$ (resp.\ $b_i = 1$)
for all $0 \leq i < m$.
\end{definition}
Note that $b_m$, which is always 1, plays no role in this construction.

The table in Figure~\ref{fig:snakes} shows 
the posets $P_0$ through $P_7$ in the third column.
Note that, in accordance with the way binary numerals are written
in drawing the snake poset of $P_n$ with $n>0$ I put the element 0 
(the ``head'' of the snake) at the right and the element $m$ at the left.
In this way one can draw the Hasse diagram of $P_n$
for $n>0$ by writing the binary expansion of $n$,
reading off the non-initial digits from left to right,
and correspondingly drawing edges from left to right,
with downward steps corresponding to 0s and upward steps corresponding to 1s.
as shown in the third column of the table.
The fourth and fifth columns, taken together,
give the {\em filters} of the poset $P_n$;
these are subsets $S$ of $P_n$ with the property that
whenever $i,j$ in $P_n$ satisfy $i \in S$ and $j \gtrdot i$,
one must have $j \in S$ as well.
(It is more common to state this property
using the partial ordering instead of the covering relation,
but the two definitions are equivalent by transitivity.)
The fourth column lists the filters that include 0
while the fifth column lists the filters that omit 0.
Finally, the sixth column gives the ratio of
the number of filters that include 0
to the number of filters that omit 0.
Anticipating the probabilistic point of view,
the column is labeled ``Probability ratio'';
it is the ratio of the probability of $E_0$ to the probability of $E_0^c$,
where $E_0$ is the set of filters that include 0
and $E_0^c$ is the set of filters that omit 0,
under the uniform probability distribution.
In keeping with probability theory nomenclature,
the symbol $E$ is chosen to suggest the word ``event''.

\begin{definition}
Let $i_n$ (resp.\ $o_n$) be the number of filters of $P_n$ 
that include (resp.\ omit) 0, so that $i_0 = 0$ and $o_0 = i_1 = o_1 = 1$.
\end{definition}

\begin{proposition}
Take $n \geq 2$ and let $n' = \lfloor n/2 \rfloor$. Then
$i_{n} = i_{n'}$ and $o_{n} = i_{n'}+o_{n'}$ when $n$ is even and
$i_{n} = i_{n'}+o_{n'}$ and $o_{n} = o_{n'}$ when $n$ is odd.
\label{prop:recur}
\end{proposition}

\begin{proof}
We exploit the fact that the snake poset 
obtained from $P_n$ by removing the rightmost vertex and edge
(``beheading'' the snake) is isomorphic to the snake poset $P_{n'}$;
one need only relabel the vertices by subtracting 1.
We call this reindexing operation {\em downshifting},
as opposed to the operation of adding 1, which we call {\em upshifting}.
\begin{figure}[h!]
\begin{center}
\includegraphics[width=6in]{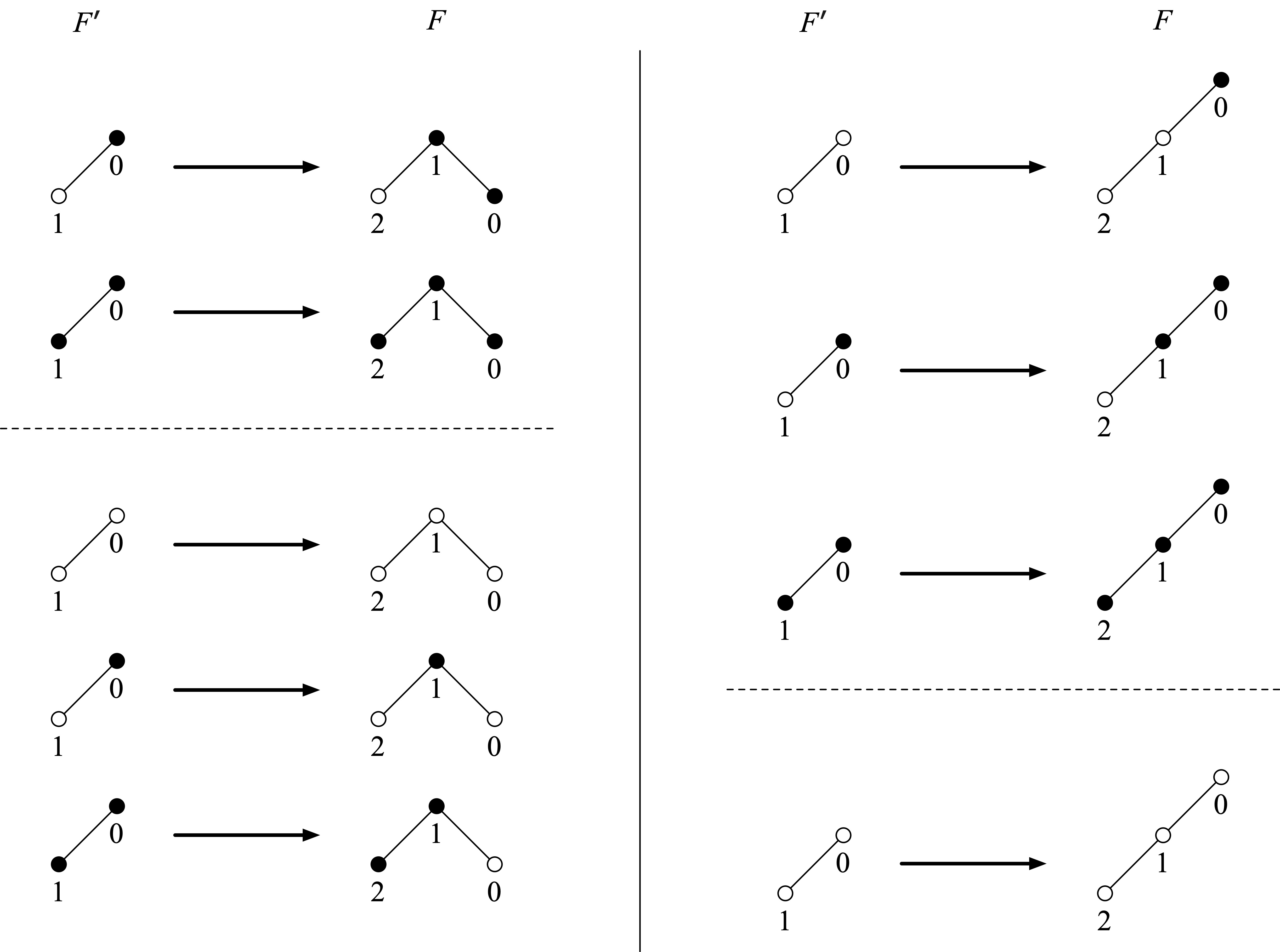}
\end{center} 
\caption{Filters for $P_6$ and $P_7$ from filters for $P_3$.}
\label{fig:snake-shift}
\end{figure}

If $n$ is even, then each filter $F$ of $P_n$ that includes 0
is obtained from a filter $F'$ of $P_{n'}$ that includes 0
by upshifting the elements and adjoining 0.
Likewise, each filter $F$ of $P_n$ that omits 0
is obtained from a filter $F'$ of $P_{n'}$ 
that may or may not include 0
by upshifting the elements (and not adjoining 0).
If $n$ is odd, then each filter $F$ of $P_n$ that includes 0
is obtained from a filter $F'$ of $P_{n'}$ 
that may or may not include 0
by upshifting the elements and adjoining 0,
while each filter $F$ of $P_n$ that omits 0
is obtained from a filter $F'$ of $P_{n'}$ that omits 0
by upshifting the elements (and not adjoining 0).
Figure~\ref{fig:snake-shift} illustrates the preceding assertions
for $n=6$ (left) and $n=7$ (right).
The claim follows.
\end{proof}

\begin{definition}
Define Dijkstra's function $\fusc(\cdot)$ by 
the initial conditions $\fusc(0) = 0$ and $\fusc(1) = 1$
and by the recurrence 
$$\fusc(n) = \left\{
\begin{array}{ll}
\fusc(n/2) & \mbox{for even $n > 1$,} \\
\fusc((n-1)/2) + \fusc((n+1)/2) & \mbox{for odd $n > 1$.}
\end{array} \right.$$
\end{definition}

\begin{proposition}
$i_n = \fusc(n)$ and $o_n = \fusc(n+1)$ for all $n \geq 0$.
\end{proposition}

\begin{proof}
By strong induction on $n$.
\end{proof}

\begin{proposition}
For each positive rational number $r$
there is a unique positive integer $n$ for which $\fusc(n)/\fusc(n+1)$
(or equivalently $i_n / o_n$) equals $r$. 
\label{prop:cw}
\end{proposition}

\begin{proof}
This is the fundamental theorem in the Calkin-Wilf sequence literature,
essentially due to Stern~\cite{S58}.
In one direction the bijection is simple: given a positive integer $n$,
the positive rational $r$ associated with $n$ is $\fusc(n)/\fusc(n+1)$.
In the other direction, suppose $r$ is a positive rational.
$r$ has two finite continued fraction expansions,
one with even length and one with odd length;
we will use the one with odd length.
E.g., $3/2$ can be written as the length-2 continued fraction $1+1/2$ 
or as the length-3 continued fraction $1+1/(1+1/1)$, so we will use the latter.
Write $r$ as $[c_1;c_2,c_3,\dots,c_m] := 
c_1 + 1/(c_2 + 1/(c_3 + \cdots + 1/c_m) \cdots )$ 
with $c_1 \geq 0$ and $c_2,...,c_m \geq 1$ with $m$ odd,
and create a binary string that, read from right to left,
consists of $c_1$ 1’s followed by $c_2$ 0’s 
followed by $c_3$ 1’s followed by~\dots~followed by $c_m$ 1’s.
This string is the binary representation of $n$.
For instance, when $r=10/7$ we have $r = 1 + 1/(2+1/3) = [1;2,3]$,
giving the bit-string string 111001 which,
interpreted as a binary expansion, gives 57,
which indeed satisfies $\fusc(57)/\fusc(58) = 10/7$.
For a full proof, see~\cite{CW00}. 
\end{proof}

\begin{definition}
For $r \geq 0$ rational,
define $S_r$ to be the poset $P_n$ for the unique integer $n \geq 0$
satisfying $r = \fusc(n)/\fusc(n+1)$
(see Proposition~\ref{prop:cw}).
Thus the posets $P_0,\dots,P_7$ shown in Figure~\ref{fig:snakes} 
will be denoted by $S_0$, $S_{1/1}$, $S_{1/2}$, $S_{2/1}$, $S_{1/3}$,
$S_{3/2}$, $S_{2/3}$, and $S_{3/1}$, respectively.
\end{definition}

Note that for $r,n$ with $S_r = P_n$,
the probability that a uniform random filter of $S_r$ includes 0 divided by
the probability that a uniform random filter of $S_r$ omits 0 is $i(n)/o(n)$
which equals $r$.

\begin{definition}
For $n \geq 0$
put $${\bf r}_n = \left( \begin{array}{cc} i_n & o_n \end{array} \right)$$
so that in particular 
${\bf r}_0 = \left( \begin{array}{cc} 0 & 1 \end{array} \right)$
and
${\bf r}_1 = \left( \begin{array}{cc} 1 & 1 \end{array} \right)$.
Let a bare ${\bf r}$ denote ${\bf r}_1$.
Define
$$L = \left( \begin{array}{cc} 1 & 0 \\ 1 & 1 \end{array} \right)
\ \ \text{and}
\ \ U = \left( \begin{array}{cc} 1 & 1 \\ 0 & 1 \end{array} \right).$$
Writing the binary expansion of $n$ as $b_m b_{m-1} \cdots b_1 b_0$,
let $T_i$ be $U$ or $L$ according to whether $b_i$ is 0 or 1.
\end{definition}

\begin{proposition}
For all $n>0$,
$${\bf r}_n = 
\left\{
\begin{array}{ll}
{\bf r}_{n'} \, U & \mbox{if $n$ is even,} \\
{\bf r}_{n'} \, L & \mbox{if $n$ is odd}
\end{array} \right.$$
so that ${\bf r}_n = {\bf r} \, T_{m-1} \cdots T_1 T_0$
(note that we exclude $T_m$).
\end{proposition}

\begin{proof}
Immediate.
\end{proof}

\begin{figure}[h!]
\begin{center}
\includegraphics[width=1.5in]{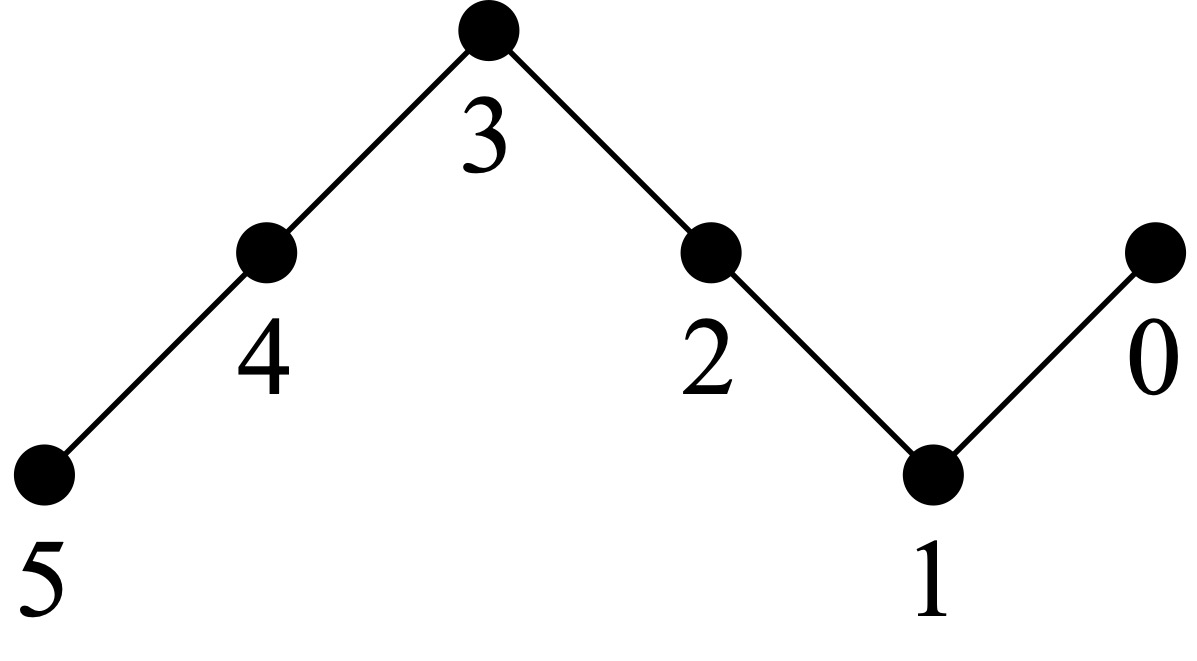}
\end{center} 
\caption{The poset $P_{57}=P_{111001_{{\rm two}}}$, aka the poset $S_{10/7}$.}
\label{fig:snake-fifty-seven}
\end{figure}

\begin{example}
Write the binary expansion of 57 as $b_5 b_4 b_3 b_2 b_1 b_0 = 111001$.
When the leading 1 is dropped we get 11001.
Replacing each 0 with $U$ and each 1 with $L$
we obtain the matrix product $LLUUL$ satisfying
$${\bf r}_1 \ L \ L \ U \ U \ L =
\left( \begin{array}{cc} 10 & 7 \end{array} \right).$$
Figure~\ref{fig:snake-fifty-seven} shows
the poset $P_{57}=P_{111001_{{\rm two}}}$, also known as $S_{10/7}$;
it has 17 filters, of which 10 contain 0 and 7 do not.
\end{example}

More generally, we can use the bits in the binary expansion of $n$
to find a product of matrices such that
multiplying that matrix on the left by the all-1's row vector of length two
gives a vector whose components
are respectively the numerator and denominator of
the $n$th entry in the Calkin-Wilf enumeration
of the positive rationals.

%

For all positive rational numbers $r$,
the snake $S_{1/r}$ can be obtained from the snake $S_r$
by replacing every upward step by a downward step and vice versa,
and the snake $S_{r+1}$ can be obtained from the snake $S_r$
by adding an upward step at the right.

\section{$q$-deformed positive rational numbers}
\label{sec:q}

\begin{definition}
Suppose $r$ is a positive rational number and $q$ is a positive real number.
Let $\cF_r$ be the set of filters in the snake poset $S_r$.
Define the {\em weight} $W_q(F)$ of a filter $F \in \cF_r$ to be $q^{|F|}$, 
let $Z = Z_{r,q} = \sum_{F \in \cF_r} W_q(F)$
(the partition function associated with the weight function $W_q$),
and let $\mu_{r,q}$ be the probability measure
that assigns $F \in \cF_r$ probability $W_q(F)/Z$.
Let $\llb r \rrb_q$ be
the probability that a $\mu_{r,q}$-random filter of $P_n$ includes 0 divided by
the probability that a $\mu_{r,q}$-random filter of $P_n$ omits 0.
(Note that setting $q=1$ assigns all filters weight 1
and so gives the uniform distribution studied in the previous section.)
\end{definition}

\begin{definition}
Let $i_n(q)$ (resp.\ $o_n(q)$) be 
the sum of the $q$-weights of the filters of $P_n$ 
that include (resp.\ omit) 0, so that for instance 
$i_0(q) = 0$, $o_0(q) = 1$, $i_1(q) = q$ and $o_1(q) = 1$.
\end{definition}

\begin{proposition}
Take $n \geq 2$ and let $n' = \lfloor n/2 \rfloor$. Then
$i_{n} = q i_{n'}$ and $o_{n} = i_{n'}+o_{n'}$ for even $n>1$ and
$i_{n} = q i_{n'} + q o_{n'}$ and $o_{n} = o_{n'}$ for odd $n>1$.
\label{prop:qrecur}
\end{proposition}

\begin{proof}
Straightforward (see the proof of Proposition~\ref{prop:recur}).
\end{proof}

\begin{proposition}
For every nonnegative rational number $r$,
$\llb r+1 \rrb_q = q \: \llb r \rrb_q \: + \: q$.
\label{prop:add}
\end{proposition}

\begin{proof}
Apply Proposition~\ref{prop:qrecur}
with $S_r = P_{n'}$, $S_{r+1} = P_{n}$, and $n=2n'+1$.
We have $i_{n} = q i_{n'} + q o_{n'}$ and $o_{n} = o_{n'}$,
so that $i_n(q)/o_n(q) = q i_{n'}(q)/o_{n'}(q) \: + \: q$.
The claim follows.
\end{proof}

\begin{proposition}
For every positive rational number $r$,
$\llb 1/r \rrb_q = 1 / \llb r \rrb_{1/q}$.
\label{prop:flip}
\end{proposition}

\begin{proof}
Suppose $S_r = P_n$ and $S_{1/r} = P_{\hat{n}}$,
and let $m+1$ be the number of elements of the poset $S_r$ as earlier,
which is also the number of elements of the poset $S_{1/r}$.
To every filter $F$ of $S_r$ corresponds a filter $F'$ 
of the dual poset $S_r^* = S_{1/r}$
obtained by flipping $F$ upside down and taking the complement.
If we use the same labels $0,1,\dots,m$ for
the elements of $F$ and $F'$, we have $F' = \{0,1,\dots,m\} \setminus F$
so that $|F'| = m+1 - |F|$. It follows that 
the $q$-weight of $F$ times the $q$-weight of $F'$ equals $q^{m+1}$;
equivalently, the $q$-weight of $F$ divided by the $q^{-1}$-weight of $F'$
equals $q^{m+1}$.
When $F$ includes 0, $F'$ omits 0, and vice versa,
so $i_n(q) / o_{\hat{n}}(1/q) = q^{m+1}$ and $o_n(q) / i_{\hat{n}}(1/q) = q^{m+1}$.
Therefore
$$\llb r \rrb_q \llb 1/r \rrb_{1/q} =
  \left( \frac{i_n(q)}{o_n(q)} \right) 
  \left( \frac{i_{\hat{n}}(1/q)}{o_{\hat{n}}(1/q)} \right) =
  \left( \frac{i_n(q)}{o_{\hat{n}}(1/q)} \right) 
  \left( \frac{i_{\hat{n}}(1/q)}{o_n(q)} \right) = q^{m+1}/q^{m+1} = 1$$
or (replacing $r$ by $1/r$) $\llb 1/r \rrb_q \llb r \rrb_{1/q} = 1$.
\end{proof}

\begin{proposition}
$\llb r \rrb_q = q \: [r]_q$ for all nonnegative rational numbers $r$.
\label{prop:q}
\end{proposition}

\begin{proof}
We know that this is true for $r=0$.
The formulas $\llb r+1 \rrb_q = q \: \llb r \rrb_q \: + \: q$
and $\llb 1/r \rrb_q = 1 / \llb r \rrb_{1/q}$ 
proved in Propositions~\ref{prop:add} and~\ref{prop:flip}
also apply to the variant $q$-deformed rational numbers $q \: [r]_q$
as consequences of the recurrence relations
$[r+1]_q = q \: [r]_q + 1$ and $[1/r]_q = 1/[r]_{1/q}$
of Morier-Genoud and Ovsienko.
Furthermore, each value of $[r]_q$ with $r$ positive and rational
can be derived from $[0]_q = 0$ by means of these two relations;
this is essentially just the continued fraction algorithm. 
Hence it follows by induction on the continued fraction complexity of $r$
that $\llb r \rrb_q = q \: [r]_q$ holds for all nonnegative rational $r$.
\end{proof}

\begin{definition}
For $n \geq 0$ put 
$${\bf r}_n (q) = 
\left( \begin{array}{cc} i_n (q) & o_n (q) \end{array} \right)$$
so that in particular
${\bf r}_0 (q) = \left( \begin{array}{cc} 0 & 1 \end{array} \right)$
and
${\bf r}_1 (q) = \left( \begin{array}{cc} q & 1 \end{array} \right)$.
Let a bare ${\bf r}$ denote ${\bf r}_1 (q)$.
Put $$L(q) = \left( \begin{array}{cc} q & 0 \\ q & 1 \end{array} \right)
\ \ \text{and}
\ \ U(q) = \left( \begin{array}{cc} q & 1 \\ 0 & 1 \end{array} \right).$$
Writing the binary expansion of $n$ as $b_m b_{m-1} \cdots b_1 b_0$,
let $T_i=T_i(q)$ be $U(q)$ or $L(q)$ according to whether $b_i$ is 0 or 1.
\end{definition}

\begin{proposition}
For all $n>0$,
$${\bf r}_n (q)= 
\left\{ \begin{array}{ll}
{\bf r}_{n'} (q) \, U(q) & \mbox{if $n$ is even,} \\
{\bf r}_{n'} (q) \, L(q) & \mbox{if $n$ is odd}
\end{array} \right.$$
so that 
${\bf r}_n (q) = {\bf r} \: T_{m-1} \cdots T_1 \, T_0$
(note that we exclude $T_m$).
\end{proposition}

\begin{proof}
Immediate.
\end{proof}

Up till now we have viewed the matrices $L(q)$ and $U(q)$ 
as tools for constructing larger snake posets from smaller ones.
At this point we adopt a different perspective
and view them as transfer matrices
the carry state information from left to right within a fixed snake matrix.
For instance, we can write the partition function for filters of $S_r$ as
$Z(r,q) = {\bf r} \: T_{m-1} \cdots T_1 \, T_0 \, {\bf c}$
where ${\bf c}$ is the column vector consisting of two 1's.
(Technically this product gives a 1-by-1 matrix, not a number.
One can fix this by introducing a single dot-product at the left or right,
but I prefer to keep the formula symmetrical and tacitly employ 
the natural identification of 1-by-1 matrices with numbers.)

\begin{example}
To compute $\llb 10/7 \rrb_q$, first determine
that the Calkin-Wilf index of 10/7 is 57
as we did in Section~\ref{sec:snakes}. 
Then write 57 in binary as 111001
and drop the leading 1 to get 11001.
Then turn that string of bits into a string of $L$'s and $U$'s
and compute $${\bf r}(q) \, L (q) \, L (q) \, U (q) \, U (q) \, L (q) = 
\left( \begin{array}{cc} q+q^2+2q^3+3q^4+2q^5+q^6 & 1+q+2q^2+2q^3+q^4 
\end{array} \right).$$ 
Lastly, divide the first element of the vector by the second element
to get $\llb \frac{10}{7} \rrb_q$.
\label{ex:tsq}
\end{example}

The factorization
$Z(r,q) = ({\bf r} \, T_{m-1} \cdots T_{k}) (T_{k-1} \cdots T_0 \, {\bf c})$
with $k$ between 1 and $m-1$ admits an interpretation
in which the row-vector ${\bf r} \: T_{m-1} \cdots T_{k}$,
viewed as the output of the left sub-snake,
becomes an input for the right sub-snake corresponding to
the product $T_{k-1} \cdots T_0 \, {\bf c}$.
That is, a snake can be viewed as a vector-processing organism 
that accepts input at its tail, 
transforms that information as it propagates through the snake, 
and produces output at its head.
Thus, two snakes can be concatenated into a single snake, 
or a single snake can be decomposed into two sub-snakes.

Note that if you wish to compare this arrangement with that of actual snakes,
you need to think in terms of the flow of sensory information (as opposed to 
the flow of nutrients which normally goes in the opposite direction).

``Pinned'' versions of the $T_i$-matrices can be used
to compute probabilities of events pertaining to
a random filter of a snake graph. Define projection matrices
$$I^{{\rm in}}  = \left(\begin{array}{cc} 1 & 0 \\ 0 & 0 \end{array} \right), \ 
\ I^{{\rm out}} = \left(\begin{array}{cc} 0 & 0 \\ 0 & 1 \end{array} \right)$$
and put
$$U^{{\rm in}}(q)  = U(q) \ I^{{\rm in}}  
= \left(\begin{array}{cc} q & 0 \\ 0 & 0 \end{array} \right), 
\ \ \ \ U^{{\rm out}}(q) = U(q) \ I^{{\rm out}} 
= \left(\begin{array}{cc} 0 & 1 \\ 0 & 1 \end{array} \right),$$
$$L^{{\rm in}}(q)  = L(q) \ I^{{\rm in}}  
= \left(\begin{array}{cc} q & 0 \\ q & 0 \end{array} \right), 
\ \ \ \ L^{{\rm out}}(q) = L(q) I^{{\rm out}} 
= \left(\begin{array}{cc} 0 & 0 \\ 0 & 1 \end{array} \right).$$
To find the sum of the weights of all the filters of a snake graph
that match a certain pattern of inclusions and omissions,
replace each relevant factor in the matrix-product
by the appropriate pinned matrix.

\begin{example}
The sum of the weights of the filters of $P_{111001_{{\rm two}}} = S_{10/7}$
that include the vertex 3 and omit the vertex 2
is (the sole entry in) the 1-by-1 matrix
$${\bf r} \, L \, L_{{\rm in}} \, 
U_{{\rm out}} \, U \, L \, {\bf c} = 
\left( q+2q^2+2q^3+4q^4 \right) . $$
One can divide this weight-sum by the partition function
(given by the unpinned version of the exact same product)
to obtain the probability of the event in question.
\end{example}


%
%
%

\section{$q$-deformed positive irrational numbers}
\label{sec:irrational}

In this section we extend the construction of the preceding section
from finite snake posets to infinite snake posets.
Throughout, $q>0$ is fixed.

\begin{definition}
Let $x>0$ be irrational, and write
\[ x = [c_1;c_2,c_3,\dots] := c_1+\frac{1}{c_2+\frac{1}{c_3+\cdots}} \]
with $c_1\ge0$ and $c_i\ge1$ for $i>1$.
Let $S_x$ be the infinite snake graph that read from right to left
consists of $c_1$ downward (leftward) steps, followed by $c_2$ upward steps,
followed by $c_3$ downward steps, followed by $c_4$ upward steps,
followed by $c_5$ downward steps, and so on. 
\end{definition}

\begin{example}
The continued fraction expansion of $\sqrt{2}$ has
$c_1 = 1$ and $c_i = 2$ for all $i > 1$.
This expansion is associated with
the snake poset shown in Figure~\ref{fig:sqrt},
which is associated with the leftward-infinite bit-string
$\cdots110011001$.

\begin{figure}[h!]
\begin{center}
\includegraphics[width=2.5in]{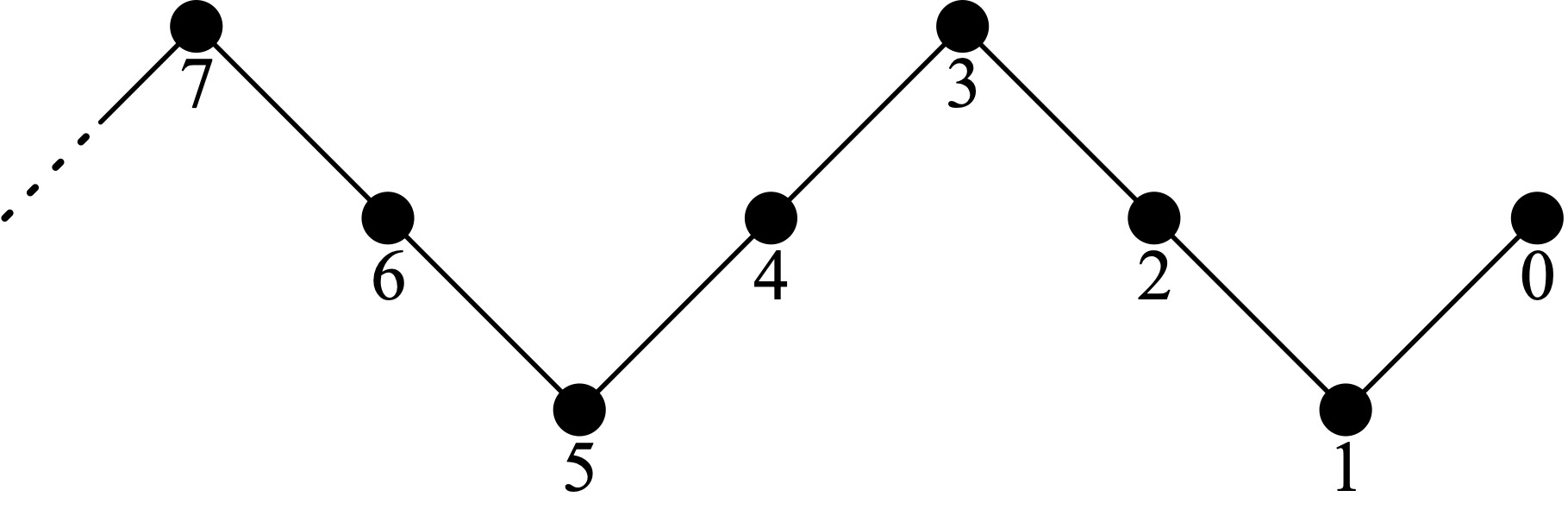}
\end{center}
\caption{The snake poset for the square root of 2.}
\label{fig:sqrt}
\end{figure}
\end{example}

Let $\cF_x$ be the set of filters of $S_x$.
Since $x$ is irrational, the continued fraction expansion of $x$
has infinitely many upward steps and infinitely many downward steps,
so $S_x$ has infinitely many peaks 
(vertices $i$ such that $i\!+\!1 \: \lessdot \: i \: \gtrdot \: i\!-\!1$).
Every subset of the set of peaks is a filter, so $\cF_x$ is uncountable.
Hence we cannot define a probability measure
$\mu_{x,q}$ on $\cF_x$ in the direct way that we used when $\cF_x$ was finite.
Instead, we take a limit of measures $\mu_{\widetilde{x},q}$
for rational numbers $\tilde{x}$ approaching $x$.
To show that the limit exists, it suffices to show that 
the probabilities of events determined by 
only finitely many vertices of $S_x$ (say vertices $0,\dots,k$)
are barely influenced by how the snake is continued 
beyond a distant vertex $\ell$, 
and that this influence tends to zero as $\ell \rightarrow \infty$
(for each fixed $k$).
Since any sequence of rationals approaching an irrational number 
has eventually constant initial partial quotients, 
and since the partial quotients determine the shape of the snake,
this will imply that any sequence of rationals approaching $x$ will do.

\begin{definition}
For $k \geq 0$, an \emph{atomic cylinder event on $0,\dots,k$}
is the event that a random filter agrees with a prescribed filter
in the induced subposet on the vertices $0,\dots,k$.
\end{definition}

\noindent

For each $i \ge 0$, let the transfer matrix $T_i$ be defined 
exactly as in Section~\ref{sec:q}; that is, $T_i$ equals $L(q)$ or $U(q)$ 
according to whether the rightward step 
from $i+1$ to $i$ is an upward step or a downward step.
For integers $a > b$, define the matrix ${}_aT_b$
(pronounced ``$T$ from $a$ to $b$'') as
\[ {}_aT_b = T_{a-1}T_{a-2}\cdots T_{b+1}T_b \]
and put ${}_aT_a=I$. Thus ${}_aT_b$
transfers state information from vertex $a$ to vertex $b$.
In this notation, $T_i = {}_{i+1}T_i$.
The composition rule ${}_aT_b\,{}_bT_c = {}_aT_c$
is immediate whenever $a \ge b \ge c$.
If we want to be maximally explicit
about the vertices and edges that information is being passed through,
we can write ${\bf r}$ as ${\bf r}_m$ 
and write ${\bf c}$ as ${}_0 {\bf c}$,
enabling us to rewrite the partition function product 
$Z(\widetilde{x},q) = {\bf r} \, T_{m-1} \cdots T_0 \, {\bf c}$ as 
$$Z(\widetilde{x},q) = {\bf r}_m \  {}_{m} T_{m-1} \  \cdots 
\  {}_1 T_0 \  {}_0 {\bf c}.$$

Choose integers $k$ and $\ell>k$ and $m>\ell$,
and let $S=S_{\widetilde{x}}$ be any finite snake on the vertices
$0,\dots,m$ that agrees with $S_x$ on the vertices $0,\dots,\ell$.
The vertices $0,\dots,k$ form the {\em proximal block},
the vertices $k+1,\dots,\ell$ form the {\em medial block},
and the vertices $\ell+1,\dots,m$ form the {\em distal block} or {\em tail}.
By the discussion that followed Example~\ref{ex:tsq},
the tail may be viewed as a vector-processing gadget
whose output is a row vector.
Denote this vector by ${}_m{\bf r}_{\ell}$.

Fix an atomic cylinder event $E$ on the proximal block.
We want to compute the effect of the tail of $S_{\tilde{x}}$
on the probability $\mu_{\tilde{x},q}(E)$.
To avoid the normalizing constant $Z_{\tilde{x},q}$
we introduce a second atomic cylinder event $E'$
and study $\mu_{\tilde{x},q}(E)/\mu_{\tilde{x},q}(E')$.

The multiplicative contribution of the proximal block
to $\mu_{\tilde{x},q}(E)$ is described by the column vector
$ {\bf s}_E = {}_kT_0^{\,E}\,{\bf c}_0 $
obtained by applying the appropriate
pinned transfer matrices to ${\bf c}_0$.
Likewise put $ {\bf s}_{E'} = {}_kT_0^{\,E'}\,{\bf c}_0 $.

The total $q$-weight of all filters satisfying $E$ is therefore
\[ Z(E) = {}_m{\bf r}_{\ell}\, {}_{\ell}T_k\, {\bf s}_E, \]
and similarly
\[ Z(E') = {}_m{\bf r}_{\ell}\, {}_{\ell}T_k\, {\bf s}_{E'}. \]
Hence
\[ \frac{\mu_{\tilde{x},q}(E)} {\mu_{\tilde{x},q}(E')} =
\frac{ {}_m{\bf r}_{\ell}\, {}_{\ell}T_k\, {\bf s}_E }
{ {}_m{\bf r}_{\ell}\, {}_{\ell}T_k\, {\bf s}_{E'} }.  \]

The vector ${}_m{\bf r}_{\ell}$ depends on the chosen continuation
of the snake beyond vertex $\ell$ (that is, on the tail).
However, the formula above uses the tail only via
the row vector ${}_m{\bf r}_{\ell}$.
Thus, in order to compare different tails,
it suffices to study the effect of multiplying arbitrary positive
row vectors by the fixed matrix ${}_{\ell}T_k$.

Replace ${}_m{\bf r}_{\ell}$ by an arbitrary positive row vector
${\bf u} = ( \ u_{\rm in} \ u_{\rm out} \ )$.
We wish to study the ratio
\[ R_{\ell,k}({\bf u}) =
\frac{ {\bf u}\ {}_{\ell}T_k\ {\bf s}_E }
{ {\bf u}\ {}_{\ell}T_k\ {\bf s}_{E'} }.  \]
Note that the ratio is unchanged if \({\bf u}\) is multiplied by 
a positive scalar; hence all that matters is the projective class 
of \({\bf u}\) (that is, $u_{\rm in} : u_{\rm out}$).
We call $u_{\rm in} : u_{\rm out}$
the {\em projective coordinate} of $\bf u$.
The row vector
${\bf u}' = {\bf u}\ {}_{\ell}T_k$ has its own projective coordinate
$u'_{\rm in} : u'_{\rm out}$.  The dependence of
\(R_{\ell,k}({\bf u})\) on the tail is controlled by 
how much the map ${\bf u}\mapsto {\bf u}'$
can change projective coordinates.
We shall show that, as \(\ell\) increases with \(k\) fixed, the possible
projective values of \( {\bf u}\ {}_{\ell}T_k \)
for positive row vectors \({\bf u}\) occupy a smaller and smaller interval.

\begin{definition}
For positive row vectors \[ {\bf u} =
(u_{\rm in}\ u_{\rm out}), \qquad {\bf v}=(v_{\rm in}\ v_{\rm out}), \]
define their projective distance by
\[ d({\bf u},{\bf v}) = 
\left| \log\frac{u_{\rm in}/u_{\rm out}} {v_{\rm in}/v_{\rm out}} \right|. \]
This is Hilbert's projective metric on the positive cone of row vectors.
\end{definition}

The effect of multiplying a row vector by $L(q)$, $U(q)$, or
$L(q)U(q)$ is particularly simple.

\begin{lemma}
Let ${\bf u}$ and ${\bf v}$ be positive row vectors.  Right multiplication
by $L(q)$ or by $U(q)$ cannot increase their projective distance, while
right multiplication by $L(q)U(q)$ decreases their projective distance by
a factor of at most
\[ \kappa(q)=\frac{1}{\bigl(\sqrt q+\sqrt{q+1}\bigr)^2}<1 .  \]
\end{lemma}

\begin{proof}
Write a positive row vector as \( {\bf u}=(u_{\rm in}\ u_{\rm out}) \)
and use \( t={u_{\rm in}}/{u_{\rm out}} \)
as its projective coordinate.  The projective distance between two
positive row vectors is the absolute value of the difference of the
logarithms of their projective coordinates.

Right multiplication by $L(q)$
sends \( t \mapsto q(t+1). \)
In logarithmic projective coordinates, the derivative is
\[ \frac{d\log(q(t+1))}{d\log t} = \frac{d(\log q(t+1))/dt}{d(\log t)/dt}
= \frac{t}{t+1}, \] which lies between $0$ and $1$.
Note however that it cannot be uniformly bounded away from 1.

Right multiplication by $U(q)$
sends \( t \mapsto qt/(t+1). \)
In logarithmic projective coordinates, the derivative is
\[ \frac{d\log(qt/(t+1))}{d\log t} = \frac{d(\log qt/(t+1))/dt}{d(\log t)/dt}
= \frac{1}{t+1}, \] which again lies between $0$ and $1$
but again cannot be uniformly bounded away from 1.

Finally,
\[ L(q)U(q)= \begin{pmatrix} q^2&q\\ q^2&q+1 \end{pmatrix}.  \]
Right multiplication by this matrix sends
\[ t \mapsto \frac{q^2(t+1)}{qt+q+1}.  \]
In logarithmic projective coordinates, the derivative is
\[ \frac{d}{d\log t} \log\left(\frac{q^2(t+1)}{qt+q+1}\right)
= \frac{t}{t+1}-\frac{qt}{qt+q+1} = \frac{t}{(t+1)(qt+q+1)}.  \]
This expression is maximized when \( t=\sqrt{(q+1)/(q)}\)
and its maximum value is
\[ \frac{1}{\bigl(\sqrt q+\sqrt{q+1}\bigr)^2}.  \]
Thus multiplication by $L(q)U(q)$ multiplies projective distance by 
$\kappa(q)$ or less, and clearly $\kappa(q) < 1$.
\end{proof}

We assume that the medial block contains at least one peak
(I explain later why this assumption is harmless).

Let \[ \mathcal C = 
\{(u_{\rm in}\ u_{\rm out}):u_{\rm in}>0,\ u_{\rm out}>0\} \]
be the positive cone of row vectors.
For a subset $\mathcal D\subseteq\mathcal C$, define its projective diameter by
\[ \operatorname{diam}(\mathcal D) =
\sup_{{\bf u},{\bf v}\in\mathcal D} \left| \log
\frac{u_{\rm in}/u_{\rm out}} {v_{\rm in}/v_{\rm out}} \right|.  \]

The following cone-contraction lemma is the key step
in the proof of existence of $\mu_{x,q}$.

\begin{lemma}
$\operatorname{diam}(\mathcal C\,{}_{\ell}T_k) \rightarrow 0$
as $\ell\rightarrow\infty$.
\end{lemma}

\begin{proof}
Consider the image \[ \mathcal C\,{}_{\ell}T_k = 
\{\,{\bf u}\,{}_{\ell}T_k:{\bf u}\in\mathcal C\,\}.  \]
Since the medial block contains at least one peak,
the word in $L(q)$ and $U(q)$ corresponding to ${}_{\ell}T_k$
contains at least one occurrence of $L(q)U(q)$.
Hence we may write \[ {}_{\ell}T_k=A\,L(q)U(q)\,B \]
for suitable products $A$ and $B$ of $L(q)$'s and $U(q)$'s.
Since $\mathcal C A\subseteq\mathcal C$,
the image $(\mathcal C A)L(q)U(q)$ is contained in
$\mathcal C L(q)U(q)$.
The latter has finite projective diameter because
$L(q)U(q)$ is a strictly positive matrix,
and because right multiplication by a strictly positive 2-by-2 matrix 
sends the positive cone into a bounded interval of the projective line.
Finally, right multiplication by $B$ cannot increase projective diameter,
so $\mathcal C\,{}_{\ell}T_k$ also has finite projective diameter.

Now let $N(k,\ell) \geq 1$ be the number of occurrences 
of $L(q)U(q)$ in the transfer word for the medial block.
Note that each peak in the medial block corresponds
to one occurrence of $L(q)U(q)$ in the transfer word.
After the first such occurrence has made the projective diameter finite,
each subsequent occurrence decreases the diameter by a further factor
of at most $\kappa(q)$, while the intervening $L(q)$'s and $U(q)$'s do not
increase it. Hence
\[ \operatorname{diam}(\mathcal C\,{}_{\ell}T_k) \le
D(q)\,\kappa(q)^{N(k,\ell)-1}, \]
where
\[ D(q)=\operatorname{diam}(\mathcal C\,L(q)U(q))<\infty .  \]

Since $x$ is irrational, the snake $S_x$ has infinitely many peaks,
and for each fixed $k$, the number of peaks in the medial block
must go to infinity as $\ell \rightarrow \infty$.
Note also that our earlier assumption that $N(k,\ell) \geq 1$
is harmless because we are sending $\ell$ to infinity.
Thus, for fixed $k$, $N(k,\ell)\rightarrow\infty$
as $\ell\rightarrow\infty$. It follows that
$\operatorname{diam}(\mathcal C\,{}_{\ell}T_k) \rightarrow 0$
as $\ell\rightarrow\infty$.
\end{proof}

\begin{corollary}
The ratio
\[ \frac{\mu_{\tilde{x},q}(E)} {\mu_{\tilde{x},q}(E')} \]
has a limit that is independent of the finite continuation of the snake
beyond vertex $\ell$.
\end{corollary}

\begin{proof}
The function
\[ {\bf w} \mapsto \frac{{\bf w}\,{\bf s}_E} {{\bf w}\,{\bf s}_{E'}} \]
depends only on the projective class of the positive row vector
${\bf w}$, and is continuous as a function on the projective line.
Since $\mathcal C\,{}_{\ell}T_k$ is eventually contained in 
a compact subinterval of the projective line,
continuity implies uniform continuity there.
Because the projective diameter of $\mathcal C\,{}_{\ell}T_k$
tends to zero, it follows that
\[ R_{\ell,k}({\bf u}) =
\frac{{\bf u}\,{}_{\ell}T_k\,{\bf s}_E}
     {{\bf u}\,{}_{\ell}T_k\,{\bf s}_{E'}} \]
becomes independent of the choice of positive row vector ${\bf u}$ as
$\ell\to\infty$.
Taking ${\bf u}={}_m{\bf r}_{\ell}$, we obtain the desired conclusion.
\end{proof}

\begin{theorem}
There exists a unique Borel probability measure $\mu_{x,q}$ on $\cF_x$ 
whose cylinder probabilities are the limits constructed above.
\end{theorem}

\begin{proof}
Choose one atomic cylinder event $E_0$ on the vertices
$0,\dots,k$.  For every atomic cylinder event $E$, define
\[ \rho(E) = \lim_{\ell\to\infty} 
\frac{\mu_{\tilde{x},q}(E)} {\mu_{\tilde{x},q}(E_0)}.  \]
Since there are only finitely many atomic cylinder events on
$0,\dots,k$, we may normalize these ratios by setting
\[ \mu_{x,q}(E) = \frac{\rho(E)} {\sum_F\rho(F)}, \]
where the sum ranges over all atomic cylinder events on
$0,\dots,k$.

These probability distributions are compatible as $k$ varies, because
they arise as limits of the compatible finite-snake distributions.
Hence they determine a finitely additive probability measure $\mu_{x,q}$ 
on cylinder sets of $S_x$.
The space of filters of $S_x$ is a closed subset of $\{0,1\}^{\mathbb N}$,
and the cylinder events generate its Borel $\sigma$-algebra.
The standard extension theorem therefore yields a unique Borel
probability measure $\mu_{x,q}$ having these cylinder probabilities.
\end{proof}

Thus we can now define our $q$-deformation of the positive real numbers:

\begin{definition}
For every irrational $x>0$, define 
\[ \llb x\rrb_q = \frac{\mu_{x,q}(0\in F)} {\mu_{x,q}(0\notin F)}. \]
\end{definition}

\begin{theorem}
Let $x>0$ be irrational, let $q>0$, and let $(r_n)$ be any sequence of
positive rational numbers converging to $x$.
Then the measures $\mu_{r_n,q}$ converge weakly to $\mu_{x,q}$,
and in particular, $\llb r_n \rrb_q \rightarrow \llb x \rrb_q$.
\label{thm:continuity}
\end{theorem}

\begin{proof}
The argument used to construct $\mu_{x,q}$
made no assumptions about the rational number $\widetilde{x}$
used to approximate $x$, other than the assumption that
as $\ell \rightarrow \infty$, the snake $S_{\widetilde{x}}$
should look like the snake $S_x$ at vertices smaller than $\ell$.
But a sequence of $\widetilde{x}$'s for which
$S_{\widetilde{x}}$ converges to $S_x$ in this sense
is nothing other than a sequence of $\widetilde{x}$'s
that converges to $x$ in the ordinary sense.
Hence the cylinder probabilities converge, 
which is precisely the definition of weak convergence 
for probability measures on $\cF_x$.
The convergence of $\llb r_n \rrb_q$ follows immediately.
\end{proof}

It is worth mentioning that for fixed $q$,
continuity of the map $x \mapsto \llb x \rrb_q$
holds only in the vicinity of irrational values of $x$.
For instance, taking $q=2$, it is not hard to show that
every rational $r<1$ satisfies $[r]_2 < 1/2$
while every rational $r>1$ satisfies $[r]_2 > 3/2$.
Indeed, it appears that for $0 < q \neq 1$,
the function that sends $x>0$ to $\llb x \rrb_q$
is a ``devil's staircase'' function,
continuous at irrational values of $x$
and discontinuous at rational values of $x$.

The continuity of $x \mapsto \llb x \rrb_q$ for irrational values of $x$ 
has several immediate consequences. First and foremost,
putting $q=1$ we have the formula $\llb x \rrb_1 = x$ which justifies 
my calling $\llb x \rrb_q$ a ``$q$-deformation of $x$'' to begin with.
Also, continuity implies that 
$\llb x \rrb_q$ satisfies the relations
$$\llb x+1 \rrb_q = q \llb x \rrb_q + q$$ and
$$\llb 1/x \rrb_q = \frac{1}{\llb x \rrb_{1/q}}$$
for all positive real numbers,
not just all positive rational numbers.

Theorem~\ref{thm:continuity} tells us that there is a different definition
of $\llb x \rrb_q$ that avoids statistical mechanics entirely:
for $x>0$ one can simply define the function $q \rightarrow \llb x \rrb_q$
as the pointwise limit of the functions $q \rightarrow \llb r_n \rrb_q$
where $r_n \rightarrow x$.
However, to show that this is definition makes sense---that is,
to show that these functions converge pointwise
to a function that is independent of the
choice of the approximating sequence of rationals---requires
nontrivial work.

Our measures $\mu_{x,q}$ are Gibbs measures.
From the point of view of statistical mechanics, 
the existence and uniqueness of these Gibbs measures are unsurprising, 
since one-dimensional lattice models with finite-range interactions 
are known to admit unique Gibbs measures. 
One could therefore appeal to the general theory 
of one-dimensional Gibbs measures (see, for example,~\cite{G11}). 
The direct argument given here has the advantage 
of constructing the Gibbs measure explicitly as the limit 
of the measures on finite snakes introduced earlier, 
making the relationship between the finite 
and infinite models completely transparent.

\section{The $q$-deformed golden ratio}
\label{sec:golden}

As an application of the poset model,
we present a derivation of $\llb \phi \rrb_q$
where $\phi = (1+\sqrt{5})/2$.

We begin with the case $q=1$.
Let $p_i$ denote the probability that
a uniformly random filter $F$ of $S_\phi$ contains vertex~$i$.
We derive two relations between $p_0$ and $p_1$.

First, we find a formula for $p_1$ 
by conditioning on whether 0 belongs to $F$.
If $0 \notin F$, then necessarily $1 \notin F$.
If $0 \in F$, then deleting the vertex~0 yields 
a poset isomorphic to the dual poset $S_\phi^*$.
By the duality principle of the previous section,
there is a natural bijection between the filters of $S_\phi$
and the filters of $S_\phi^*$
under which a vertex belongs to a filter of $S_\phi$
if and only if the corresponding vertex is absent from
the corresponding filter of $S_\phi^*$.
Hence \[ p_1 = p_0(1-p_0).  \]

Next, we find a formula for $p_0$
by conditioning on whether 1 belongs to $F$.
If $1 \in F$, then necessarily $0 \in F$.
If $1 \notin F$, then the status of $0$ is unconstrained,
so at $q=1$ it is equally likely to be present or absent.
Hence \[ p_0 = p_1 + \frac12(1-p_1) = \frac{1+p_1}{2}, \]
or equivalently, \[ p_1 = 2p_0-1.  \]

Combining this last equation with $p_1=p_0(1-p_0)$ gives
$p_0^2+p_0-1=0$,
whose unique solution in $(0,1)$ is $p_0=\frac{\sqrt5-1}{2}=\frac1\phi$.
Therefore the odds $\frac{p_0}{1-p_0}$ is equal to $\phi$.

Now let $q$ be arbitrary.
Write $p_0(q)$ as $p_0$ for short
and write $p_0(1/q)$ as $\bar p_0$.
Write
\[ O=\frac{p_0}{1-p_0}=\llb\phi\rrb_q, \qquad
\bar O=\frac{\bar p_0}{1-\bar p_0} =\llb\phi\rrb_{1/q}.  \]

First, condition on whether 0 belongs to $F$.
If $0 \notin F$, then necessarily $1 \notin F$.
If $0 \in F$, then deleting vertex~0 from the poset $S_\phi$
yields a poset isomorphic to the dual poset $S_\phi^*$.
Under this bijection, the weight of a filter is transformed 
by replacing $q$ by $1/q$ while the in-versus-out status
of vertex~1 is reversed, so the conditional probability that $1 \in F$
is $1-\bar p_0$. Hence \[ p_1=p_0(1-\bar p_0).  \]

Next, condition on whether 1 belongs to $F$.
If $1 \in F$, then necessarily $0 \in F$.
If $1 \notin F$, then vertex~0 may be either present or absent,
with weights $q$ and $1$, respectively.
Hence $\Pr(0\in F\mid 1\notin F)=\frac{q}{1+q}$,
so \[ p_0 = p_1 + (1-p_1)\frac{q}{1+q}.  \]

Solving the preceding equation for $p_1$ gives
\[ p_1=(1+q)p_0-q.  \]
Combining this with our earlier equation for $p_1$, we obtain
\[ p_0(1-\bar p_0)=(1+q)p_0-q, \]
or equivalently, $p_0 \bar p_0 = q(1-p_0)$.
Substituting $p_0=\frac{O}{1+O}$ and $\bar p_0 = \frac{\bar O}{1+\bar O}$
we obtain $$\frac{O\bar O}{1+\bar O}=q$$ or equivalently $$O=q+\frac{q}{\bar O}.$$
Replacing $q$ by $1/q$ yields the companion equation
$\bar O=\frac1q+\frac{1}{qO}$.
Substituting this last equation into the one before gives
\[ O = q+\frac{q^2O}{1+O}.  \]
Clearing denominators gives $O(1+O)=q(1+O)+q^2O$,
so $O^2+(1-q-q^2)O-q=0$.
Since $O>0$, we conclude that
\[ \llb\phi\rrb_q = \frac{q+q^2-1+\sqrt{(1-q-q^2)^2+4q}}{2}.  \]

This expression is equal to $q$ times Morier--Genoud--Ovsienko's $q$-golden ratio,
in agreement with the expectation that $\llb\phi\rrb_q$
should equal $q[\phi]_q$ wherever both are defined.

\section{A dimer-model interpretation}
\label{sec:dimer}

The snake posets we introduced above are intimately related
to snake graphs as defined in~\cite{P01}.
A snake graph is a graph obtained by gluing together even-sided polygons 
to form a finite or infinite chain,
where each polygon is glued to the next along an edge
and no other gluings are performed.
We assume there is a first polygon in the chain containing 
a special edge $e$, which we will assume is
the upper horizontal edge of the rightmost hexagon.
Let that polygon have its vertices colored alternately black and white
so that in the counterclockwise orientation of the polygon
the special edge points from black to white.
Then all polygons in the chain inherit the two-coloring,
and each edge between neighboring polygons in the chain
can be assigned the letter $L$ or the letter $U$
according to the way in which the shared edge is colored.
To be more precise, suppose $P$ and $P'$ are adjacent polygons,
with $P$ nearer the tail of the snake and $P'$ nearer the head.
Then the shared edge between $P$ and $P'$
is marked with an $L$ (resp.\ a $U$) if, 
when one moves from the interior of $P$
to the interior of $P'$ by crossing the shared edge,
the black endpoint of that edge
is on one's left (resp.\ one's right).
Figure~\ref{fig:hex-fifty-seven-again}
shows the hexagonal snake graph with the code $LLUUL$,
associated with the rational number $r=10/7=[1;2,3]$.
The general rule is that the rational number $r = [c_1;c_2,c_3,\dots,c_m]$
with $m$ odd and $c_1 \geq 0$ and $c_2,c_3,\dots,c_m > 0$
corresponds to a hexagon snake that, read from right to left, 
has $c_1$ downward steps, followed by $c_2$ upward steps,
followed by $c_3$ downward steps, followed by $c_4$ upward steps,
followed by $c_5$ downward steps, \dots, followed by $c_m-1$ downward steps.
(Subtracting 1 from $c_m$ corresponds to the way we ignore the initial 1
in the binary expansion of the Calkin-Wilf index of $r$.)

\begin{figure}[h!]
\begin{center}
\includegraphics[width=3in]{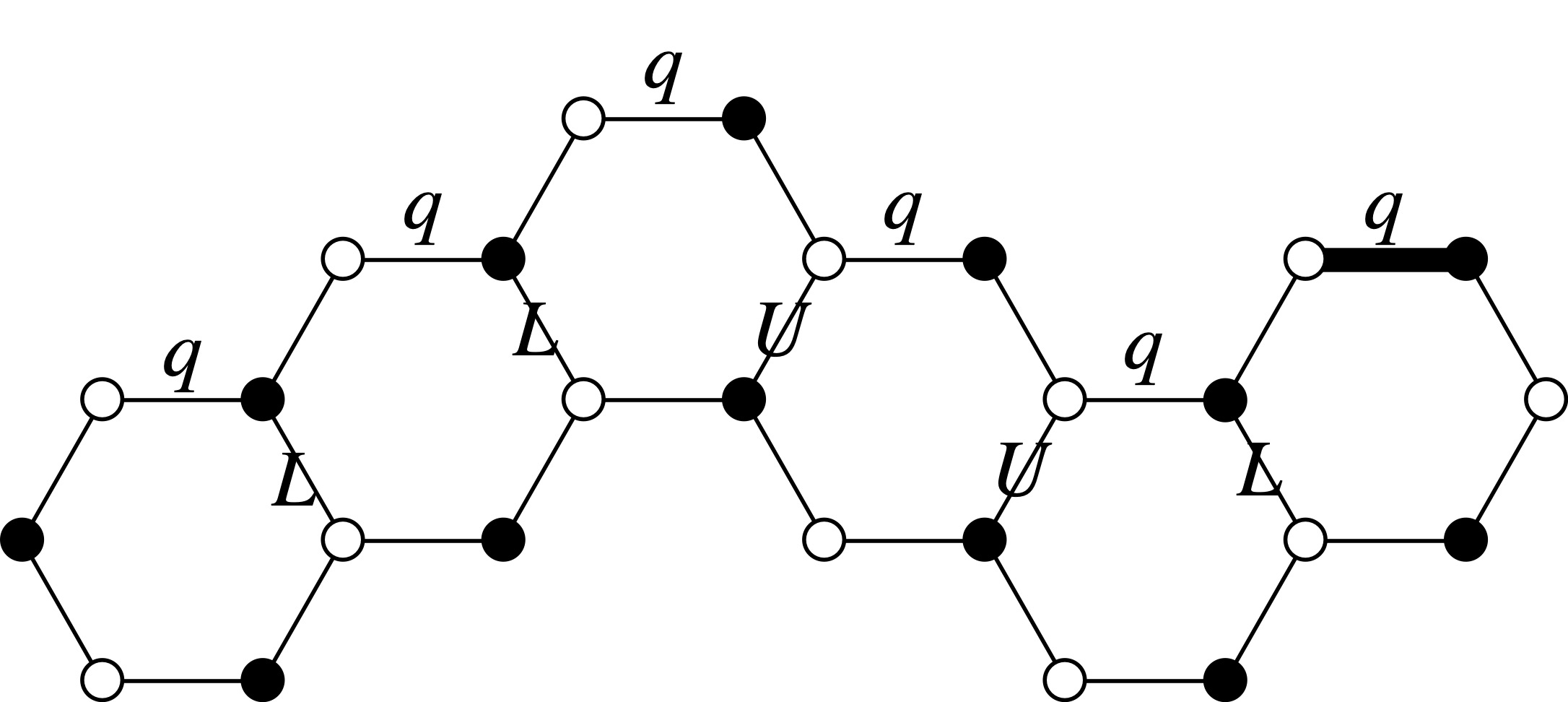}
\end{center}
\caption{The hexagon snake graph $G_{10/7}(q) = G_{[1;2,3]}(q)$.}
\label{fig:hex-fifty-seven-again}
\end{figure}

The connection between $q$-deformed rationals and dimer models
has been explored by others using square snakes, but I find 
that it is more simply expressed in terms of hexagon snakes.
To turn the snake poset $S_r$ into the corresponding hexagon snake graph $G_r$,
replace each poset element by a unit hexagon with two horizontal edges
and identify adjoining diagonal edges 
as shown in Figure~\ref{fig:hex-fifty-seven-again}
for $S_{10/7}$ (compare with Figure~\ref{fig:snake-fifty-seven}).
Each horizontal edge is either the upper edge of a hexagon
or the lower edge of a hexagon.
We assign each horizontal edge of the first type weight $q$
and assign every other edge weight $1$;
this gives a weighted graph $G_r (q)$.
The weighting determines a probability distribution
on the (perfect) matchings of the graph in the usual way.
The event $E_0$ from the poset context
corresponds to the event in which the rightmost horizontal edge
belongs to a $q$-random matching.
In the 10/7 example, the reader can check
that the matchings that include this edge
have total weight $q+q^2+2q^3+3q^4+2q^5+q^6$
while the matchings that omit this edge
have total weight $1+q+2q^2+2q^3+q^4$,
giving the probability ratio $\llb 10/7 \rrb_q$.

\begin{figure}[h!]
\begin{center}
\includegraphics[width=5in]{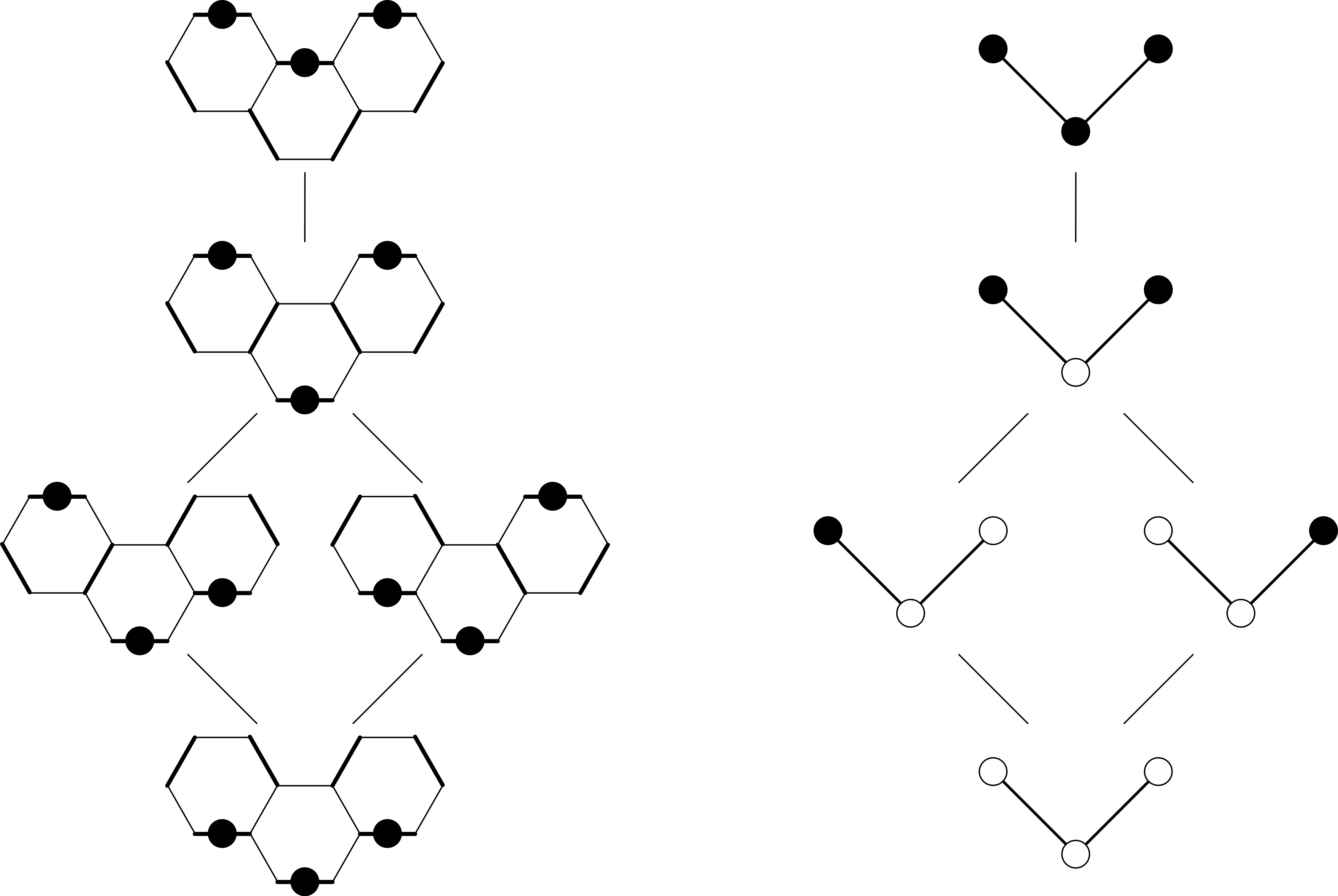}
\end{center}
\caption{Matchings of $G_{3/2}$ and filters in $S_{3/2}$.}
\label{fig:particles}
\end{figure}

It would be possible to give a self-contained treatment 
of dimers on snake graphs using the transfer-matrix formalism
without discussing posets at all,
but the theory of moves presented in~\cite{P25} gives another route.
The results proved there imply in particular that
every dimer cover of a snake graph can be obtained from every other
by means of face-moves in which three edges that border a hexagon
(one horizontal, one of positive slope, and one of negative slope)
are replaced by the other three edges that border the same hexagon.
Such a move conserves the number of horizontal edges.
We may imagine particles that occupy
the midpoints of the horizontal edges that belong to the dimer cover,
so that the effect of a face-move is that
the particle jumps from one horizontal edge to another,
as shown on the left side of Figure~\ref{fig:particles}.
This particle model serves as a bridge between
the dimer model and the poset model
without the machinery of height-functions:
the presence or absence of a particle
along the upper edge of a hexagon in a dimer cover of $G_r$
corresponds to the presence or absence
of the associated poset element in a filter of $S_r$
(see the right side of the Figure)
and both have the same $q$-weight.
That is, there is a weight-preserving one-to-one correspondence
between filters $F$ of $S_r$ and matchings $M$ of $G_r$
where an element of $S_r$ belongs to $F$
if and only if the corresponding edge $e$ in $G_r$ belongs to $M$.

\begin{figure}[h!]
\begin{center}
\includegraphics[width=2.2in]{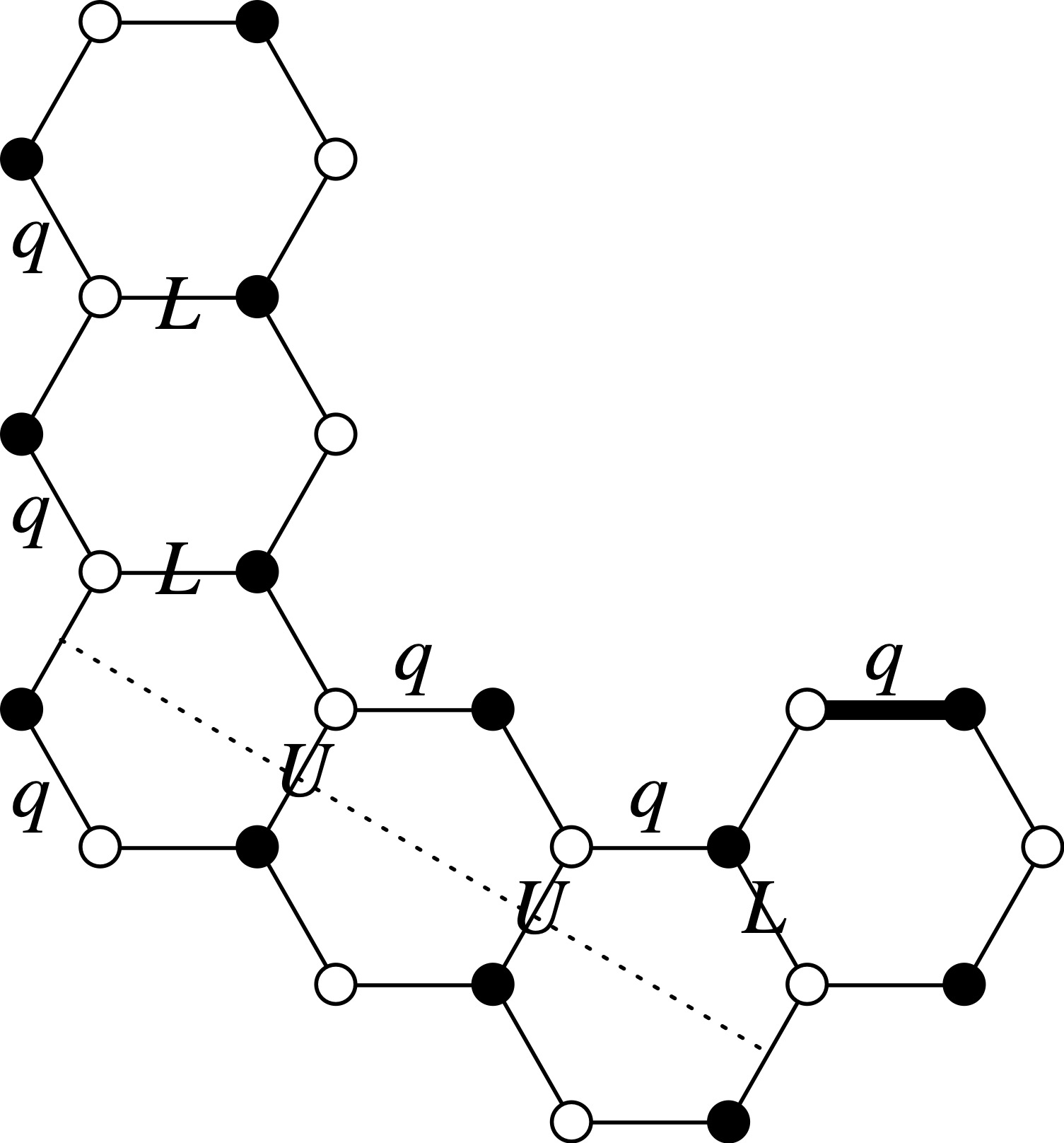}
\end{center}
\caption{Another hexagon snake graph for $r=10/7$.}
\label{fig:another-hex-fifty-seven}
\end{figure}

Figure~\ref{fig:another-hex-fifty-seven} shows 
a different snake graph with the code $LLUUL$.
It is obtained from the weighted graph $G_{10/7}(q)$ 
shown in Figure~\ref{fig:hex-fifty-seven-again}
by reflecting the leftmost part of the snake
across the dashed diagonal line, including the edge-weights.
Although the modified weighted graph is not isomorphic 
to the weighted graph $G_{10/7}(q)$
it still has the property that the odds 
in favor of inclusion of the special edge is $\llb 10/7 \rrb_q$.
To see why, notice that when $u$ and $v$ are the two endpoints
of the edge of a snake that joins hexagon $H$ to hexagon $H'$,
every matching of the snake matches $u$ with $v$ 
or matches $u$ and $v$ with their neighbors in $H$
or matches $u$ and $v$ with their neighbors in $H'$.
Consequently, flipping the part of the snake graph
that lies on one side of edge $uv$ induces a bijection
between the matchings of the flipped graph
and the matchings of the unflipped graph,
and by flipping the $q$-weights we make
the bijection weight-preserving as well.

The dimer representation also applies in the infinite limit.
For $x>0$ irrational, say $x=[c_1;c_2,c_3,c_4,\dots]$,
let $G_x(q)$ be the hexagonal snake graph that, read from right to left, 
has $c_1$ downward steps, followed by $c_2$ upward steps,
followed by $c_3$ downward steps, followed by $c_4$ upward steps,
and so on. Let the special edge $e$ once again be
the upper horizontal edge of the rightmost hexagon.
Then $\llb x \rrb_q$ is the odds of inclusion of the special edge
under the natural probability measure $\mu_{x,q}$
on the set of perfect matchings of a half-infinite snake graph $G_x(q)$.

\begin{figure}[h!]
\begin{center}
\includegraphics[width=6in]{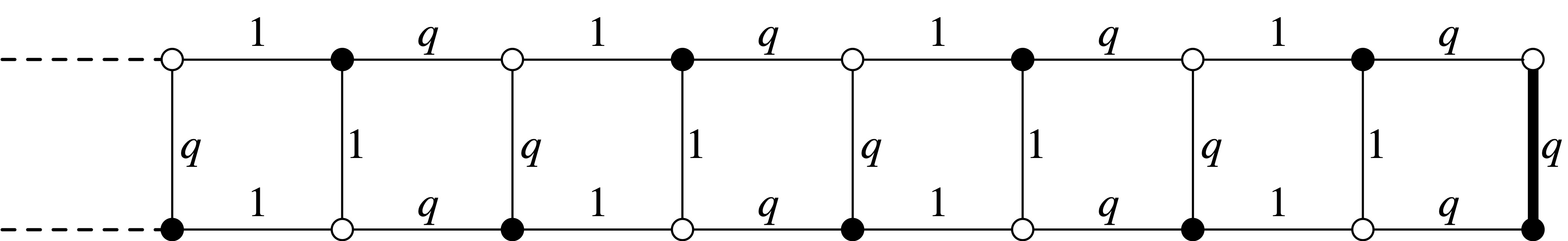}
\end{center}
\caption{A square snake graph for the golden ratio.}
\label{fig:squares}
\end{figure}

In the case where $x$ is the golden ratio $\phi=[1;1,1,1,1,\dots]$, 
it is appealing to replace the weighted snake graphs made of hexagons
by a weighted snake graph made of squares,
since this replaces a zigzagging chain of hexagons
by a straight chain of squares, as shown in Figure~\ref{fig:squares}.
Just as in the case of our two $r=10/7$ graphs,
this transformation is not a graph isomorphism,
but it leaves the number of perfect matchings and 
the statistics of random perfect matchings unaffected,
although the weighting-function gets transformed in a complicated way.
The reader can check that under the weighting shown,
replacing two horizontal edges that surround a square plaquette 
by two vertical edges that surround the same plaquette
increases the weight by a factor of $q$ --
exactly as is the case for the height function
associated with the dimer configuration.
With $q=1$, we find that the probability that
a random perfect matching of the graph
contains the rightmost edge is $\phi^{-1}$, giving the odds
$(\phi^{-1})/(1-\phi^{-1}) = \phi = \llb \phi \rrb_1$.

Figure~\ref{fig:e-snake}
shows a snake graph associated with the constant 
$e=[2;1,2,1,1,4,1,1,6,\dots]$;
the odds that a random perfect matching of this graph
contains the marked edged is $e$.

\begin{figure}[h!]
\begin{center}
\includegraphics[width=5in]{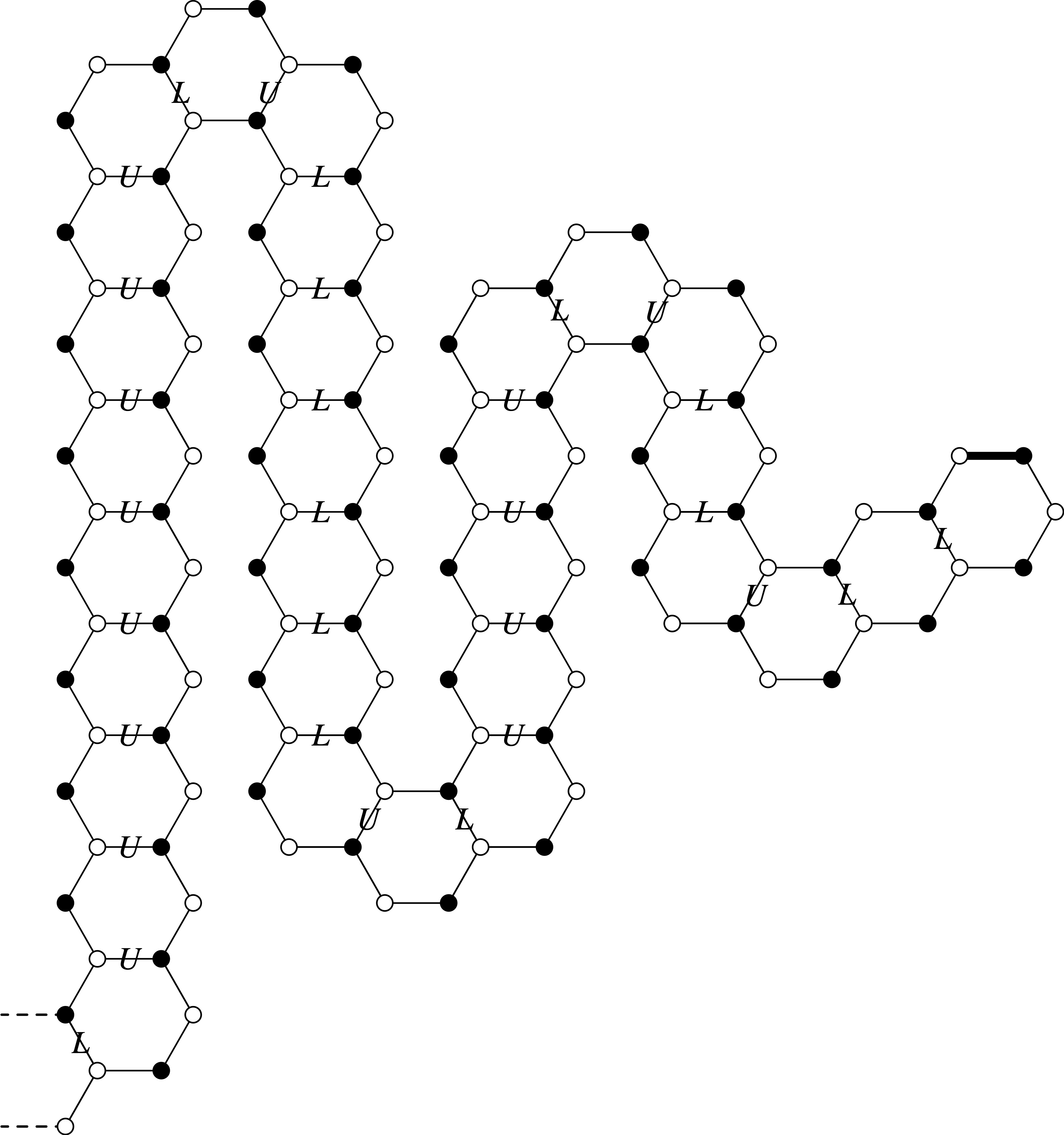}
\end{center}
\caption{A hexagon snake graph for $e$.}
\label{fig:e-snake}
\end{figure}

\section{Directions for future work}
\label{sec:future}

Perhaps the most important open question is
the relationship between $[x]_q$ and $\llb x \rrb_q$.
The former is defined via Laurent expansions
for all $q$ lying in some disk centered at 0,
and it has been shown by Etingof~\cite{E25}
that the function has an essential singularity
that prevents it from being extended
to arbitrarily large disks.
On the other hand, $\llb x \rrb_q$
is well-defined for all positive real $q$.
One might for instance hope that
wherever both $[x]_q$ and $\llb x \rrb_q$ are defined
one has $\llb x \rrb_q = q \: [x]_q$.

\bigskip

The approach taken here applies to the ``right'' $q$-rationals,
also known as the ``sharp'' $q$-rationals;
however, for each rational number $r$ there is a second deformation
(the ``left'' or ``flat'' $q$-deformation)
that also turns up in the purely algebraic theory~\cite{BBL23}.
It would be good to include the left-rationals in a fuller picture.
Also missing are the $q$-deformations of real numbers $x<0$;
one might bring them into the theory
by allowing negative powers of the matrices $L^{(q)}$ and $R^{(q)}$
(drawing inspiration from~\cite{P01}).
Note however that, in contrast with Theorem~\ref{thm:continuity},
$\llb r_n \rrb_q$ need not converge at all when $q$ is negative;
for instance, when the $r_n$ are the usual convergents
to the golden ratio and $q=-1$, the values of $\llb r_n \rrb_q$ 
form the 3-periodic (non-convergent) sequence
$-1,0,\infty,-1,0,\infty,\dots$.

There is a sense in which each positive rational $r$
acquires not two but three $q$-deformations
from the stat-mech perspective. 
The first of them arises from the finite snake $S_r$
constructed in section~\ref{sec:snakes},
equipped with the weighting described in section~\ref{sec:q}.
The second arises by prepending the half-infinite word $\cdots LLU$
to the transfer word for $S_r$;
the third arises by prepending the half-infinite word $\cdots UUL$
to the transfer word for $S_r$.
There is a sense in which the second $q$-deformation
is infinitesimally less than $\llb r \rrb_q$
while the third $q$-deformation is infinitesimally greater.
Two further infinite snakes of this kind are given by
the words $\cdots LLL$ and $\cdots UUU$,
respectively associated with $\infty$ and its formal reciprocal $1/\infty$.

\bigskip

The paper~\cite{MPS25} explores
the interpretation of the finite stat mech models of section~\ref{sec:q}
in terms of hyperbinary expansions of positive integers.
The authors give a different interpretation of $q$-weight
in terms of digit patterns in those expansions,
and briefly explore other ways of assigning weight.
For instance, one can deform $L$ and $U$ into
$$\left( \begin{array}{cc} 1 & 0 \\ r & s \end{array} \right) \mbox{\ and \ }
\left( \begin{array}{cc} r & s \\ 0 & 1 \end{array} \right)$$
for formal indeterminates $r,s$.
Such 2-by-2 matrices are
well-suited to being recast as transfer matrices in snake graphs,
where each filter corresponds to a monomial
whose exponents give more refined information about the filter
(more specifically, information about adjacent poset elements).
The approach applied in this paper should extend readily
into this broader context. 

Likewise, it should be possible to replace 2-by-2 matrices
by $k$-by-$k$ matrices with $k>2$,
along the line of development followed in~\cite{BOSZ24}.

\bigskip

{\sc Acknowledgment:} 
Both the writing of this article
and the research that went into it relied extensively
on ChatGPT and hence indirectly on the work of many
anonymous researchers whose work ChatGPT was trained on. 
The author retains full responsibility for the correctness
of all mathematical details and acknowledgments of prior
work. I thank Nicholas Ovenhouse and Valentin Ovsienko for 
helpful comments on earlier drafts of the article.


\end{document}